\documentclass[12pt]{article}

\usepackage{amsfonts, amsmath, amsthm, amssymb}

\usepackage{indentfirst}
\usepackage[T2A]{fontenc}
\usepackage[cp1251]{inputenc}
\usepackage[russian, english]{babel}

\usepackage[unicode]{hyperref}

\newcommand{\mes}{\operatorname{mes}}
\newcommand{\bigzero}{\text{\rm\Large0}}

\newtheorem{theorem}{Theorem}
\newtheorem{lemma}{Lemma}

\title{Distribution of real algebraic integers
}
\author{Denis V. Koleda}
\date{}

\begin{document}
\maketitle

\begin{abstract}
In the paper, we study the asymptotic distribution of real algebraic integers of fixed degree as their naive height tends to infinity. Let $I \subset \mathbb{R}$ be an arbitrary bounded interval, and $Q$ be a sufficiently large number. We obtain an asymptotic formula for the number of algebraic integers $\alpha$ of fixed degree $n$ and naive height $H(\alpha)\le Q$ lying in $I$. In this formula, we estimate the order of the error term from above and below. We show that the real algebraic integers of degree $n$ are distributed asymptotically like the real algebraic numbers of degree $(n-1)$ as the upper bound $Q$ of heights tends to infinity. 
\end{abstract}

\section{Introduction and main results}

In this paper we investigate the distribution of algebraic integers of arbitrary fixed degree.
We establish an asymptotic formula counting algebraic integers in an arbitrary interval (Theorem \ref{thm-main} below).
In fact, we prove that the algebraic integers of degree $n$ and height at most $Q$ tend to be distributed in the real line almost like the algebraic numbers of degree $(n-1)$ with the same heights as $Q\to \infty$ (Theorem \ref{thm-limit} below). The scheme \cite{Ka15} of a proof was published in 2015.
Note that for measuring algebraic integers we use the naive height
called simply \emph{height}.

Let $p(x) = a_n x^n + \ldots + a_1 x + a_0$ be a polynomial of degree $n$,
and let $H(p)$ be its height defined as $H(p) = \max_{0\le i\le n} |a_i|$.
Let $\alpha\in\mathbb{C}$ be an algebraic number. The \emph{minimal polynomial} of $\alpha$ is defined as a nonzero polynomial~$p$ of the minimal degree $\deg(p)$ with integral coprime coefficients and positive leading coefficient such that $p(\alpha)=0$.
For the algebraic number $\alpha$, its degree $\deg(\alpha)$ and its height $H(\alpha)$ are defined as the degree and the height of the corresponding minimal polynomial.
An algebraic number is called an \emph{algebraic integer} if its minimal polynomial is monic, that is, has the leading coefficient 1.

We denote by $\# S$ the number of elements in a finite set $S$,
$\mes_{k} S$ denotes the $k$--dimensional Lebesgue measure of a set $S \subset \mathbb{R}^{d}$ ($k \le d$).
The length of an interval $I$ is denoted by $|I|$.
The Euclidean norm of a vector $\mathbf{x}\in\mathbb{R}^k$ is denoted by $\|\mathbf{x}\|$.
To denote asymptotic relations between functions, we use the Vinogradov symbol $\ll$:
expression $f \ll g$ denotes that $f \le c\,g$, where $c$ is a constant depending on the degree $n$ of algebraic numbers.
Expression $f \asymp g$ is used for asymptotically equivalent functions, that is, $g \ll f \ll g$.
Notation $f \ll_{x_1,x_2,\ldots}\, g$ implies that the implicit constant depends only on parameters $x_1,x_2,\ldots$. Asymptotic equivalence $f \asymp_{x_1,x_2,\ldots}\, g$ is defined by analogy.

Some interesting features of the distribution of algebraic integers can be described in terms of regular systems. The idea of regular systems arose in a natural way as a useful tool for calculation of the Hausdorff dimension in the paper by Baker and Schmidt~\cite{BakSch70}, who proved that the real algebraic numbers constitute a regular system. In 1999, Beresnevich~\cite{Bere99} improved parameters of regularity \cite{BakSch70} omitting logarithmic factors. In 2002, Bugeaud~\cite{Bugeaud2002} proved that the real algebraic integers also form a regular system. In simple language, that is, there exists a constant $c_n$ depending on $n$ only such that for any interval $I\subseteq[-1,1]$ for all sufficiently large $Q\ge Q_0(I)$ there exist at least
$c_n |I| Q^n$
algebraic integers $\alpha_1,\dots,\alpha_k \in I$ of degree $n$ and height at most $Q$ such that the distances between them are at least
$Q^{-n}$.

Additionally, there are a number of papers concerning the asymptotic number of algebraic numbers $\alpha$ of given degree $n$ and multiplicative Weil height $\mathcal{H}(\alpha)\le X$ in number field extensions as $X$ tends to infinity.
When the ground field is $\mathbb{Q}$,
the height $\mathcal{H}(\alpha)$ can be expressed in the Mahler measure as $\mathcal{H}(\alpha) = M(\alpha)^{1/n}$.
Such asymptotic formulas are obtained for algebraic numbers by Masser and Vaaler \cite{MasVaa2008}, \cite{MasVaa2007}.
In 2001, Chern and Vaaler \cite[Theorem 6]{ChernVaaler2001} proved asymptotic estimates for the number of integral monic polynomials of degree $n$ having the Mahler measure bounded by $T$, which tends to infinity.
Let $\mathcal{O}_n(\mathbb{Q},X)$ denote the set of algebraic integers $\alpha$ of degree $n$ over $\mathbb{Q}$ and of Weil height $\mathcal{H}(\alpha) \le X$.
From \cite{ChernVaaler2001} we have immediately
\[
\#\mathcal{O}_n(\mathbb{Q},X) = c(n) X^{n^2} + O\!\left(X^{n^2 - 1}\right),
\]
where $c(n)$ is an explicit positive constant; in the big-O-notation the implicit constant depends only on $n$. Note that here $X$ has order of $Q^{1/n}$, where $Q$ is the upper bound for corresponding naive heights.
In 2013, Barroero \cite{Barroero2013} extended this result to arbitrary ground number fields.
Some references for this subject can be found in the book by Lang \cite[chapter 3, \S 5]{Lang1983}.
Note that these results dealing with the Weil height do not overlap with ours.

Besides, there are a series of works about counting units, i.e. invertible elements, in the ring of algebraic integers of a number field. Here we can refer to \cite{GyPe1980}, \cite{EvLo1993} (which use a~height function different from the ones mentioned above), and to \cite[chapter 3, \S 5]{Lang1983}. There one can find further references about invertible elements in rings of algebraic integers and related questions.

In 1971, H.~Brown and K.~Mahler \cite{BroMah1971} proposed a natural generalization of the Farey sequences for algebraic numbers of higher degrees.
This generalization is based on the naive height.
For a long time, the whole picture of the distribution for arbitrary fixed degrees remained unknown even for the real algebraic numbers of the second degree. In 1985, K. Mahler noted this in his letter to V. Sprind\v{z}uk.

For the real algebraic numbers of the second degree, a result of such a type was obtained in \cite{Ka2013-3}, and for an arbitrary fixed degree in \cite{Ka12} and \cite{Ko-v13-1} (a full proof with some generalizations can be found in \cite{Koleda2014a}).
In \cite{Ka12}, the author showed that the number of algebraic numbers of degree $n$ and naive height at most $Q$ lying in an interval $I$ 
equals to
\[
\frac{Q^{n+1}}{2 \zeta(n+1)} \int_I \phi_n(t) \,dt + O\left(Q^n (\ln Q)^{\delta(n)}\right),
\]
where $\zeta(\cdot)$ is the Riemann zeta function;
the function $\phi_n(t)$ is defined by the formula:
\begin{equation}\label{eq-phi-def}
  \phi_n(t) = \int\limits_{G_n(t)} \left|\sum_{k=1}^n k p_k t^{k-1}\right|\,dp_1\ldots\,dp_n, \qquad t \in \mathbb{R},
\end{equation}
where
\begin{equation}\label{eq-G-def}
G_n(t) = \left\{(p_1,\ldots,p_n) \in \mathbb{R}^n : \ \max\limits_{1\le i\le n}|p_i| \le 1, \ \left| \sum_{k=1}^n p_k t^k \right| \le 1 \right\}.
\end{equation}
In the remainder term, the implicit constant in the big-O notation depends on the degree~$n$ only. The power of the logarithm is equal to:
\[
\delta(n) =
\begin{cases}
1, & n\le 2,\\
0, & n\ge 3.
\end{cases} 
\]

Let $\mathcal{O}_n$ denote the set of algebraic integers of degree $n$.
For a set $S \subseteq \mathbb{R}$,
let $\Omega_n(Q,S)$ be the number of algebraic integers $\alpha \in S$ of degree $n$ and height at most $Q$:
\[
\Omega_n(Q,S) := \#\left\{\alpha \in \mathcal{O}_n \cap S : H(\alpha)\le Q \right\}.
\]

Note that the algebraic integers of degree 1 are simply the rational integers, which are nowhere dense in the real line. Therefore, we assume $n\ge 2$.

We prove the following two theorems.
\begin{theorem}\label{thm-main}
For any interval $I\subseteq\mathbb{R}$ we have:
\begin{equation}\label{eq-main}
\Omega_n(Q,I) = Q^n \int_I \omega_n(Q^{-1}, t) \,dt + O\!\left(Q^{n-1}(\ln Q)^{\delta(n)}\right),
\end{equation}
where the function $\omega_n(\xi, t)$ can be written in the form:
\begin{equation}\label{eq-omega-def}
\omega_n(\xi, t) = \int\limits_{D_n(\xi, t)} \left|n\xi t^{n-1} + \sum_{k=1}^{n-1} kp_k t^{k-1} \right| \,dp_1 \dots dp_{n-1},
\end{equation}
with $D_n(\xi, t) = \left\{(p_1,\dots,p_{n-1})\in\mathbb{R}^{n-1} : |p_i|\le 1, \ \left|\xi t^n + \sum_{k=1}^{n-1} p_k t^k \right| \le 1\right\}$.

In the remainder term, the implicit constant depends only on the degree $n$.
Besides, there exist intervals, for which the error of this formula has the order $O(Q^{n-1})$.
\end{theorem}

Actually, for $n=2$ it can be proved \cite{Koleda2016} that in \eqref{eq-main} instead of $O(Q\ln Q)$ we have
\[
- 2Q \int\limits_{I\cap[-Q,Q]} \frac{dt}{\max(1,|t|)} + O\left(Q\right).
\]

\begin{theorem}\label{thm-limit}
Let $n\ge 2$ be a fixed integer. For all $t\in\mathbb{R}$:
\[
\lim\limits_{\xi\to 0} \omega_n(\xi,t) = \phi_{n-1}(t),
\]
where the function $\phi_n(t)$ is defined by~\eqref{eq-phi-def}.

Besides, for fixed $n\ge 3$, when $\xi < 1$,
\begin{equation}\label{eq-omega-psi}
|\omega_n(\xi, t) - \phi_{n-1}(t)| \ll_n \begin{cases}
\xi^2 t^2, & |t|\le \kappa_1(n) \xi^{-1/2},\\
\xi, & \kappa_1(n) \xi^{-1/2} < |t| \le \kappa_2(n)\xi^{-1},\\
t^{-2}, & \kappa_2(n)\xi^{-1} < |t|,
\end{cases}
\end{equation}
where the positive quantities $\kappa_1(n)$, $\kappa_2(n)$ and the implicit constant in the Vinogradov symbol depend on $n$ only.
\end{theorem}

The inequality \eqref{eq-omega-psi} shows that for all $t$ we have the estimate $|\omega_n(\xi, t) - \phi_{n-1}(t)| \ll_n \xi$, that is, the function $\omega_n(\xi,t)$ uniformly converges to $\phi_{n-1}(t)$ as $\xi$ tends to zero.
Thus, in fact, Theorem \ref{thm-limit} could say that the limit density function of real algebraic integers of degree $n$ is equal to the density function of real algebraic numbers of degree~${(n-1)}$.
However, the following statement is true.
\begin{theorem}\label{thm-idiff}
Let $n\ge 2$. As $\xi\to 0$
\begin{equation*}
\int\limits_{-\infty}^{+\infty} \left(\omega_n(\xi, t) - \phi_{n-1}(t)\right) dt = 2^n + O(\xi^{1/2}),
\end{equation*}
where the implicit big-O-constant depends only on $n$.
\end{theorem}

Note the following interesting (and a bit surprising) feature. If we take any finite fixed interval $I$, we get by Theorem \ref{thm-limit}
\[
\lim_{\xi\to 0} \frac{\int_I \omega_n(\xi,t) dt}{\int_I \phi_{n-1}(t) dt} = 1.
\]
But if, for example, $I=\mathbb{R}$ we readily obtain from Theorem \ref{thm-idiff}
\[
\lim_{\xi\to 0} \frac{\int_{\mathbb{R}} \omega_n(\xi,t) dt}{\int_{\mathbb{R}} \phi_{n-1}(t) dt} = \frac{\gamma_{n-1}+2^n}{\gamma_{n-1}} > 1,
\]
where $\gamma_n = \int_{\mathbb{R}} \phi_n(t) dt > 0$.

The paper is divided into the following sections.
Section \ref{sec-aux} contains auxiliary statements, and the reader can skip it.
In Section \ref{sec-mainproof} we prove Theorem \ref{thm-main}.
Section \ref{sec-limit} is devoted to the proof of Theorems \ref{thm-limit} and \ref{thm-idiff}.

\section{Auxiliary statements}\label{sec-aux}

\begin{lemma}[\cite{Chela1963}]\label{lm-red-m-pol}
Let $\mathcal{R}_n^{*}(Q)$ denote the set of reducible monic integral polynomials of degree $n$ and height at most $Q$. Then
\[
\lim\limits_{Q\to \infty} \frac{\#\mathcal{R}_n^{*}(Q)}{Q^{n-1}} = \upsilon_n, \qquad
\lim\limits_{Q\to \infty} \frac{\#\mathcal{R}_2^{*}(Q)}{2Q \ln Q} = 1,
\]
where $\upsilon_n$ is an effective positive constant depending on $n$ only, $n \ge 3$.
\end{lemma}

\begin{lemma}[\cite{Dav51_LP}]\label{lm-int-p-num}
Let $\mathcal{D}\subset \mathbb{R}^d$ be a bounded region formed by points $(x_1,\dots,x_d)$ satisfying a finite collection of algebraic inequalities
\[
F_i(x_1,\dots,x_d)\ge 0, \qquad 1\le i\le k,
\]
where $F_i$ is a polynomial of degree $\deg F_i \le m$ with real coefficients.
Let
\[
\Lambda(\mathcal{D}) = \mathcal{D}\cap \mathbb{Z}^d.
\]
Then
\[
\left|\#\Lambda(\mathcal{D}) - \mes_d \mathcal{D}\right| \le C \max(\bar{V}, 1),
\]
where the constant $C$ depends only on $d$, $k$, $m$; the quantity $\bar{V}$ is the maximum of all $r$--dimensional measures of projections of $\mathcal{D}$ onto all the~coordinate subspaces obtained by making $d-r$ coordinates of points in $\mathcal{D}$ equal to zero, $r$ taking all values from $1$ to $d-1$, that is,
\[
\bar{V}(\mathcal{D}) := \max\limits_{1\le r < d}\left\{ \bar{V}_r(\mathcal{D}) \right\}, \quad
\bar{V}_r(\mathcal{D}) := \max\limits_{\substack{\mathcal{J}\subset\{1,\dots,d\} \\ \#\mathcal{J} = r}}\left\{ \mes_r \operatorname{Proj}_{\mathcal{J}} \mathcal{D} \right\},
\]
where $\operatorname{Proj}_{\mathcal{J}} \mathcal{D}$ is the orthogonal projection of $\mathcal{D}$ onto the coordinate subspace formed by coordinates with indices in $\mathcal{J}$.
\end{lemma}

\begin{lemma}\label{lm-jacob}
Let $n\ge 2$. Let $\xi$ be a fixed positive real number. Let vectors $(a_n, \dots, a_1, a_0)$ and $(b_{n-2},\dots, b_1, b_0, \alpha, \beta)$ be related by the equality
\begin{equation}\label{eq-p-roots}
\xi x^n + \sum\limits_{k=0}^{n-1} a_k x^k =
(x-\alpha)(x-\beta)\left(\xi x^{n-2}+ \sum\limits_{m=0}^{n-3} b_m x^m \right).
\end{equation}
Then this relation can be expressed in the following matrix form:
\begin{equation}\label{eq_zamena2_matr}
\left(\begin{array}{l}
\xi\\
a_{n-1}\\
a_{n-2}\\
\vdots\\
a_1\\
a_0
\end{array}\right) =
\left(\begin{array}{cccc}
1 & & \phantom{\ddots} & \bigzero\\
-(\alpha+\beta) & 1 & \phantom{\ddots} &\\
\alpha\beta & -(\alpha+\beta) & \ddots &\\
 & \alpha\beta & \phantom{\ddots} & 1\\
 & & \ddots & -(\alpha+\beta) \\ 
\bigzero & & \phantom{\ddots} & \alpha\beta 
\end{array}\right) \cdot
\left(\begin{array}{l}
\xi \\
b_{n-3}\\
\vdots\\
b_1\\
b_0
\end{array}\right),
\end{equation}
and the Jacobian of this coordinate change is equal to
\begin{equation*}
\det J = \left| \frac{\partial (a_{n-1},\ldots,a_2,a_1,a_0)}{\partial (b_{n-3},\ldots,b_0;\alpha,\beta)} \right| = 
(\beta-\alpha)\cdot g(\mathbf{b}, \alpha) \cdot g(\mathbf{b}, \beta),
\end{equation*}
where $g(\mathbf{b}, x) := \xi x^{n-2} + b_{n-3} x^{n-3} + \ldots + b_1 x + b_0$.
\end{lemma}
Note that for $n=2$ the Jacobian equals to $\xi^2(\beta-\alpha)$.

Lemma \ref{lm-jacob} can be proved in the way Lemmas 1, 2, 3 of \cite{Ko-v13-1} are proved.

\begin{lemma}\label{lm-2roots}
Let $I = [a,b)\subset \mathbb{R}$ be a finite interval, $|I|\le 1$, and let $0<\xi \le 1$.
Let $\mathcal{M}_n(\xi, I)$ be the set of polynomials $p\in\mathbb{R}[x]$ with height $H(p) \le 1$ and $\deg(p(x) - \xi x^n) < n$ that have at least $2$ roots in~$I$.
Then
\begin{equation*}
\mes_{n} \mathcal{M}_n(\xi, I) \le \lambda(n) \left(\xi + \rho^{-3}\right)^2 |I|^3,
\end{equation*}
where $\rho = \max(1, |a+b|/2)$, and $\lambda(n)$ is a constant depending only on $n$.
\end{lemma}

\begin{proof}
To simplify notation, we use $\mathcal{M} := \mathcal{M}_n(\xi, I)$.
Estimate from above the measure
\begin{equation*}
\mes_n \mathcal{M} = \int\limits_{\mathcal{M}} d{\bf a}.
\end{equation*}
Every polynomial $p(x)$ in $\mathcal{M}$ can be expressed in the form
\[
p(x) = \xi x^n + a_{n-1} x^{n-1} + \ldots + a_0 =
(x-\alpha)(x-\beta)(\xi x^{n-2}+ b_{n-3} x^{n-3}+\ldots + b_0),
\]
where $\alpha, \beta \in I$.

Change the coordinates by \eqref{eq_zamena2_matr}.
The condinition ${\bf a} \in \mathcal{M}$ is equivalent to the system of inequalities
\begin{equation}\label{eq_grM2}
\left\{
\begin{array}{l}
|a_{n-1}| = |b_{n-3} - (\alpha+\beta) \xi| \le 1,\\
|a_{n-2}| = |b_{n-4}-(\alpha+\beta)b_{n-3}+\alpha\beta \xi| \le 1,\\
|a_k| = |b_{k-2} - (\alpha+\beta) b_{k-1} + \alpha\beta b_k| \le 1, \ \ \ k = 2,\ldots,n-3,\\
|a_1| = |- (\alpha+\beta) b_0 + \alpha\beta b_1| \le 1,\\
|a_0| = |\alpha b_0| \le 1,\\
a\le \alpha < b, \\
a\le \beta < b.
\end{array}
\right.
\end{equation}

From Lemma \ref{lm-jacob}, we obtain
\begin{equation}\label{eq-M-ineq}
\mes_n \mathcal{M} \le \int\limits_{\mathcal{M}^*} |\alpha-\beta| \cdot |g(\mathbf{b}, \alpha)\; g(\mathbf{b},\beta)|\, d\mathbf{b}\, d\alpha\, d\beta,
\end{equation}
where $\mathcal{M}^*$ is the new integration domain defined by the inequlities \eqref{eq_grM2}, here $g(\mathbf{b},x) = \xi x^{n-2}+\ldots+b_1 x + b_0$. Note that here we have inequality instead of equality. The reason is that a polynomial having $k>2$ roots in $I$ can be written in $\binom{k}{2}$ different ways in the form~\eqref{eq-p-roots}.

Write the multiple integral \eqref{eq-M-ineq} in the following manner:
\[
\mes_n \mathcal{M} \le \int\limits_{I\times I} |\alpha-\beta|\, d\alpha\, d\beta \int\limits_{\mathcal{M}^*(\alpha,\beta)} |g(\mathbf{b},\alpha) g(\mathbf{b},\beta)|\,d\mathbf{b},
\]
where $\mathcal{M}^*(\alpha,\beta)$ is the set of vectors $\mathbf{b} \in \mathbb{R}^{n-1}$ that satisfy the inequalities~\eqref{eq_grM2}.

Estimate the internal integral using upper bounds for $\mes_{n-2} \mathcal{M}^*(\alpha,\beta)$ and for the function $G({\bf b},\alpha,\beta):= g(\mathbf{b},\alpha)\; g(\mathbf{b},\beta)$ for $\mathbf{b} \in \mathcal{M}^*(\alpha,\beta)$. Consider the two cases.

1) Let $|a+b|/2 \le 1$.

Estimate the measure $\mes_{n-2} \mathcal{M}^*(\alpha,\beta)$ using the submatrix equation from~\eqref{eq_zamena2_matr}:
\begin{equation}\label{eq-matr1}
\left(\begin{array}{l}
a_{n-1} - c_1\xi\\
a_{n-2} - c_0\xi\\
\vdots\\
a_3\\
a_2\\
\end{array}\right) =
\left(\begin{array}{lllll}
c_2 & & & & \bigzero\\
c_1 & c_2 & &\\
c_0 & c_1 & \ddots &\\
\vdots & \vdots & \ddots & c_2\\
 & & \dots & c_1 & c_2
\end{array}\right) \cdot
\left(\begin{array}{l}
b_{n-3}\\
b_{n-4}\\
\vdots\\
b_1\\
b_0
\end{array}\right),
\end{equation}
where $c_0 = \alpha\beta$, $c_1 = -(\alpha + \beta)$, $c_2 = 1$.

The image of the region $\mathcal{M}^*(\alpha, \beta)$ in the coordinates $(a_{n-1},\dots,a_3,a_2)$ obtained by multiplication by the matrix in \eqref{eq-matr1} is contained in a parallelepiped of the unit volume. The determinant of the matrix \eqref{eq-matr1} is equal to 1.
Hence, we have the upper bound for the measure:
\[
\mes_{n-2} \mathcal{M}^*(\alpha,\beta) \le 1.
\]

Estimate $|G(\mathbf{b},\alpha,\beta)|$ from above for $\mathbf{b}\in \mathcal{M}^*(\alpha,\beta)$. For this sake, find upper bounds of the coordinates $b_{n-3}, \dots, b_1, b_0$ in the region $\mathcal{M}^*(\alpha, \beta)$.

Since $|a+b|/2 \le 1$ and $0 < b-a \le 1$, we have the following estimate for the matrix coefficients in \eqref{eq-matr1}:
\[
\max\limits_{0\le i\le 2}|c_i| = O(1).
\]
Starting from $b_{n-3}$, we express the coefficients $b_i$ and obtain by induction
\begin{equation*}
\max\limits_{0\le i\le n-3} |b_i| \ll_n 1.
\end{equation*}
Hence, we have that $|G({\bf b},\alpha,\beta)| \ll_n 1$ for all $\mathbf{b}\in \mathcal{M}^*(\alpha, \beta)$.
So we obtain for $\alpha,\beta\in [a,b)$
\[
\int\limits_{\mathcal{M}^*(\alpha,\beta)}|G({\bf b},\alpha,\beta)|\,d{\bf b} \ \ll_n \ 1.
\]
Therefore, for $|a+b|/2 \le 1$ we have
\begin{equation}
\mes_n \mathcal{M} \ll_n |I|^3.
\end{equation}

2) Let $|a+b|/2 > 1$.

Since \eqref{eq_zamena2_matr} we have
\begin{equation}\label{eq-matr2}
\left(\begin{array}{l}
a_{n-3}\\
a_{n-4}\\
\vdots\\
a_1\\
a_0
\end{array}\right) =
\left(\begin{array}{lllll}
c_0 & c_1 & \dots & & \\
 & c_0 & \ddots & \vdots & \vdots \\
 & & \ddots & c_1 & c_2\\
 & & & c_0 & c_1\\
\bigzero & & & & c_0
\end{array}\right) \cdot
\left(\begin{array}{l}
b_{n-3}\\
b_{n-4}\\
\vdots\\
b_1\\
b_0
\end{array}\right),
\end{equation}
where $c_0 = \alpha\beta$, $c_1 = -(\alpha + \beta)$, $c_2 = 1$.

Hence, we obtain the upper bound
\[
\mes_{n-2} \mathcal{M}^*(\alpha,\beta) \le |\alpha\beta|^{-n+2}.
\]

Find an upper bound for $|G(\mathbf{b},\alpha,\beta)|$. For this sake, estimate $|b_i|$ from above.

The matrix coefficients $c_i$ in \eqref{eq-matr2} can be estimated in the following manner:
\[
c_0 \asymp \rho^2, \quad c_1 = O(\rho), \quad c_2 = 1,
\]
where $\rho = |a+b|/2$. Here, we take into account that $0< b-a \le 1$.

By induction, we estimate from above $|b_i|$. From \eqref{eq-matr2}, we have
\begin{align*}
|b_0| &=  |c_0|^{-1} |a_0| = O(\rho^{-2}), \\
|b_1| &= |c_0|^{-1} |a_1 - c_1 b_0| = O(\rho^{-2})\; O(1 + \rho^{-1}) = O(\rho^{-2}).
\end{align*}
Proceeding by induction for $i=2,3,\dots,n-3$, we obtain
\[
|b_i| = |c_0|^{-1} |a_i - c_1 b_{i-1} - c_2 b_{i-2}| = O(\rho^{-2})\; O(1 + \rho^{-1} + \rho^{-2}) = O(\rho^{-2}),
\]
where in the big-O-notation the implicit constants depend only on $n$.

Hence, for any $x\in [a,b)$ we obtain $|g(\mathbf{b},x)| = O(\xi \rho^{n-2} + \rho^{n-5})$, so
\[
|G(\mathbf{b},\alpha,\beta)| = O\left(\rho^{2(n-2)} \left(\xi+\rho^{-3}\right)^2\right).
\]
Therefore, we obtain
\[
\mes_n \mathcal{M} \le \lambda(n) \left(\xi+\rho^{-3}\right)^2 |I|^3.
\]
The lemma is proved.
\end{proof}

\begin{lemma}[\cite{Ka12}]\label{lm-empty-int}
Let $x_0=a/b$ with $a\in\mathbb{Z}$, $b\in\mathbb{N}$ and $\gcd(a,b)=1$.
Then there are no algebraic numbers $\alpha$ of degree $\deg \alpha = n$ and height $H(\alpha)\le Q$ in the interval $|x-x_0|\le r_0$, where
\[
r_0=r_0(x_0, Q)=\frac{\kappa(n)}{b^n Q},
\]
and $\kappa(n)$ is an effective constant depending only on $n$.

For a neighborhood of infinity: no algebraic number $\alpha$ of degree $\deg(\alpha)=n$ and height $H(\alpha)\le Q$ lies in the set $\{x\in\mathbb{R}: |x|\ge Q+1\}$.
\end{lemma}

\begin{lemma}\label{lm-diff}
Let $V\subset\mathbb{R}^k$ be a bounded region being symmetric with respect to the origin of coordinates. Let $\mathbf{v} = (v_1,\dots,v_k)$ be a fixed nonzero vector, and let $\epsilon > 0$ be a real number.
Then for any $0<\lambda<1$
\[
(1-\lambda)\,\epsilon \cdot \mes_k V(\lambda\,\epsilon) \ \le \ 
 \int\limits_V |\mathbf{v}\cdot\mathbf{x} + \epsilon|\,d\mathbf{x} - 
 \int\limits_V |\mathbf{v}\cdot\mathbf{x}|\,d\mathbf{x} \ \le \ \epsilon \cdot \mes_k V(\epsilon),
\]
where
\[
V(\epsilon) := \left\{\mathbf{x}\in V : |\mathbf{v}\cdot\mathbf{x}| < \epsilon\right\}.
\]
If $V(\epsilon)=V$, then
\[
\int\limits_V |\mathbf{v}\cdot\mathbf{x} + \epsilon|\,d\mathbf{x} = \epsilon \cdot \mes_k V.
\]
\end{lemma}

\begin{proof}
Changing $\mathbf{x}$ for $-\mathbf{x}$ in the integral, after transformation, we obtain
\[
\int\limits_V |\mathbf{v}\cdot\mathbf{x} + \epsilon|\,d\mathbf{x} = \int\limits_V |\mathbf{v}\cdot\mathbf{x} - \epsilon|\,d\mathbf{x} =
\int\limits_V \frac{|\mathbf{v}\cdot\mathbf{x} + \epsilon|+|\mathbf{v}\cdot\mathbf{x} - \epsilon|}{2}\,d\mathbf{x}.
\]
Since
\[
\frac{|\mathbf{v}\cdot\mathbf{x} + \epsilon|+|\mathbf{v}\cdot\mathbf{x} - \epsilon|}{2} =
\begin{cases}
|\mathbf{v}\cdot\mathbf{x}|, & |\mathbf{v}\cdot\mathbf{x}| \ge \epsilon,\\
\epsilon, & |\mathbf{v}\cdot\mathbf{x}| < \epsilon,
\end{cases}
\]
we have
\[
\int\limits_V |\mathbf{v}\cdot\mathbf{x} + \epsilon|\,d\mathbf{x} - 
 \int\limits_V |\mathbf{v}\cdot\mathbf{x}|\,d\mathbf{x} \ =
\int\limits_{V(\epsilon)} \left(\epsilon - |\mathbf{v}\cdot\mathbf{x}|\right)\,d\mathbf{x}.
\]
The lemma is proved.
\end{proof}

\begin{lemma}\label{lm-mes-diam}
Let $\mathbf{a}=(a_1,\dots, a_k)$ and $\mathbf{b}=(b_1,\dots,b_k)$ be fixed noncollinear vectors.
Let a region $V \subset \mathbb{R}^k$ be defined by the inequalities
\[
\begin{cases}
|\mathbf{a}\cdot \mathbf{x}| \le H_1,\\
|\mathbf{b}\cdot \mathbf{x}| \le H_2,
\end{cases}
\]
where $\mathbf{x} = (x_1,\dots,x_k)\in\mathbb{R}^k$ is the radius vector.

Then the area of the section $\mathcal{S}$ of the region $V$ by the linear span of $\mathbf{a}$ and $\mathbf{b}$ is equal~to
\[
\mes_2 \mathcal{S} = \frac{4 H_1 H_2}{\sqrt{\mathbf{a}^2 \mathbf{b}^2 - (\mathbf{a}\cdot \mathbf{b})^2}},
\]
and the diameter of the section satisfies the inequality
\[
\frac{\sqrt{\mathbf{b}^2 H_1^2 +\mathbf{a}^2 H_2^2}}{\sqrt{\mathbf{a}^2 \mathbf{b}^2 - (\mathbf{a}\cdot \mathbf{b})^2}}
\le \operatorname{diam} \mathcal{S} \le
\frac{\|\mathbf{b}\| H_1 + \|\mathbf{a}\| H_2}{\sqrt{\mathbf{a}^2 \mathbf{b}^2 - (\mathbf{a}\cdot \mathbf{b})^2}}.
\]
\end{lemma}

\begin{proof}
Let a parallelogram be formed by intersection two strips with widths $h_1$ and $h_2$ and an angle $\alpha$ between them. Then its area equals to
\[
S = \frac{h_1 h_2}{\sin\alpha},
\]
and its diameter is equal to
\[
d = \frac{\sqrt{h_1^2 + h_2^2 - 2 h_1 h_2 \cos \alpha}}{\sin \alpha}.
\]
For the section $\mathcal{S}$, we have
\[
h_1 = \frac{2H_1}{\|\mathbf{a}\|}, \quad
h_2 = \frac{2H_2}{\|\mathbf{b}\|}, \quad
\sin \alpha = \frac{\sqrt{\|\mathbf{a}\|^2 \|\mathbf{b}\|^2 - (\mathbf{a}\cdot \mathbf{b})^2}}{\|\mathbf{a}\| \|\mathbf{b}\|},
\]
hence, we obtain the lemma.
\end{proof}

\section{The proof of the main theorem}\label{sec-mainproof}

Let $I = [\alpha, \beta)$ be a finite interval. Denote by $\mathcal{N}_n(Q,k,I)$ the number of irreducible integer monic polynomials of degree $n$ and height at most $Q$ having exactly $k$ roots in~$I$. It is easy to see that
\begin{equation}\label{eq-Omega-N}
\Omega_n(Q,I) = \sum_{k=1}^n k\, \mathcal{N}_n(Q,k,I).
\end{equation}
Let $\mathcal{G}_n(\xi,k,S)$ be the set of real polynomials of degree $n$ and height at most~1 with the leading coefficient $\xi$ having exactly $k$ roots (with respect to multiplicity) in a set $S$.
From Lemmas \ref{lm-red-m-pol} and \ref{lm-int-p-num}, we have
\begin{equation}\label{eq-N-G}
\mathcal{N}_n(Q,k,I) = Q^n \mes_n \mathcal{G}_n(Q^{-1},k,I) + O\!\left(Q^{n-1} (\ln Q)^{\delta(n)}\right),
\end{equation}
where in the big-O-notation the implicit constant depends only on $n$.

It is easy to see that for all $0<\xi\le 1$ the function
\begin{equation}\label{eq-Omega-G}
\widehat{\Omega}_n(\xi, S):= \sum_{k=1}^n k \mes_n \mathcal{G}_n(\xi,k,S)
\end{equation}
is additive and bounded on the set of all subsets $S\subseteq \mathbb{R}$.

Let us show that $\widehat{\Omega}_n(\xi, I)$ can be written as the integral of a continuous function over~$I$.
Let
\begin{equation*}
\mathcal{B}(\xi, I) = \left\{{\bf p}\in\mathbb{R}^{n+1} : \deg(p(x)-\xi x^n) < n, \  p(\alpha)p(\beta)< 0, \ H(p)\le 1 \right\},
\end{equation*}
where ${\bf p} = (\xi, p_{n-1},\ldots,p_1,p_0)$ is the vector of the coefficients of the polynomial $p(x)= \xi x^n+\ldots+p_1 x+p_0$, and $\xi = Q^{-1}$.
Obviously, every polynomial from $\mathcal{B}(\xi,I)$ has the odd number of roots in the interval~$I$.

In Lemma \ref{lm-2roots}, we have $\mathcal{M}_n(\xi, I) = \bigcup_{k=2}^n\mathcal{G}_n(\xi, k, I)$. Hence, it follows that
\begin{equation}\label{eq-Omega-D}
\widehat{\Omega}_n(\xi, I) = \mes_n \mathcal{B}(\xi, I) + O(|I|^3),
\end{equation}
where in the big-O-notation the implicit constant depends only on $n$.

Now we calculate
\[
\mes_n \mathcal{B}(\xi, I) = \int\limits_{\mathcal{B}(\xi, I)} dp_0\,dp_1 \dots dp_{n-1}.
\]
The region $\mathcal{B}(\xi, I)$ can be defined by the following inequalities
\begin{equation}\label{eq-sys}
\begin{cases}
\max\limits_{0\le k\le n-1} |p_k| \le 1,\\
f_*(p_1,\dots,p_{n-1},\xi) \le p_0 \le f^*(p_1,\dots,p_{n-1},\xi),
\end{cases}
\end{equation}
where
\begin{align*}
f_*(p_1, \dots, p_{n-1}, p_n) :=& \min\left\{-\sum_{k=1}^n p_k\alpha^k,\  -\sum_{k=1}^n p_k\beta^k \right\},\\
f^*(p_1, \dots, p_{n-1}, p_n) :=& \max\left\{-\sum_{k=1}^n p_k\alpha^k,\  -\sum_{k=1}^n p_k\beta^k \right\}.
\end{align*}

Define
\begin{equation*}
h(\xi,p_{n-1},\ldots,p_1):= f^*(p_1,\dots,p_{n-1},\xi) - f_*(p_1,\dots,p_{n-1},\xi)
\end{equation*}
and consider the regions
\[
D_* := D_n(\xi, \alpha)\cap D_n(\xi, \beta), \qquad
D^* := D_n(\xi, \alpha)\cup D_n(\xi, \beta).
\]
where
\[
D_n(\xi, t) := \left\{(p_1,\dots,p_{n-1})\in\mathbb{R}^{n-1} : \max_{1\le i\le n-1} |p_i|\le 1, \ \left|\xi t^n + \sum_{k=1}^{n-1} p_k t^k \right| \le 1\right\}.
\]
The inequalities $|f_*| \le 1$ and $|f^*| \le 1$ hold for all $(p_1,\dots,p_{n-1})\in D_*$. For any $(p_1,\dots,p_{n-1})\not\in D^*$, the inequalities $|f_*| > 1$ and $|f^*| > 1$ hold simultaneously, and so for sufficiently close $\alpha$ and $\beta$ the system of inequalities \eqref{eq-sys} is contradictory. Otherwise, the inequalities $f_* < -1$ and $f^* > 1$ would hold, and thus we would have $h(\xi, p_{n-1},\dots,p_1) > 2$. However the function $h$ uniformly tends to zero for all $(p_1,\dots,p_{n-1})\in[-1,1]^{n-1}$ as $\alpha$ and $\beta$ converge.

Therefore,
\begin{multline*}
\int\limits_{D_*} h(\xi,p_{n-1},\ldots,p_1)\,dp_{n-1}\ldots dp_1 \le \\
\le \ \mes_n\mathcal{B}(\xi, I) \ \le \\
\le \int\limits_{D^*} h(\xi,p_{n-1}\ldots,p_1)\,dp_{n-1}\ldots dp_1.
\end{multline*}
Hence, it follows that
\[
\left|\mes_n\mathcal{B}(\xi, I) - \int\limits_{D_n(\xi, \alpha)} h(\xi,p_{n-1},\ldots,p_1)\,dp_{n-1}\ldots dp_1 \right| \le 
\int\limits_{D^*\setminus D_*} h(\xi,p_{n-1}\ldots,p_1)\,dp_{n-1}\ldots dp_1.
\]
It is easy to show that the difference of $D^*$ and $D_*$ has a small measure for sufficiently close $\alpha$ and $\beta$:
\[
\mes_{n-1} (D^* \setminus D_*) = O(\beta-\alpha).
\]
Thus, as $\beta\to\alpha$, we obtain for all $\alpha\in\mathbb{R}$
\begin{equation*}
\mes_n \mathcal{B}(\xi, I)=\omega_n(\xi,\alpha)|I|+o(|I|),
\end{equation*}
where $\omega_n(\xi,t)$ is defined in \eqref{eq-omega-def}.

Hence, as $|I|\to 0$, from \eqref{eq-Omega-D} we obtain that
\[
\widehat{\Omega}_n(\xi,I) = \omega_n(\xi,\alpha)|I|+o(|I|),
\]
So we have
\[
\widehat{\Omega}_n(\xi,I)  = \int_I\omega_n(\xi,t)\,dt.
\]
Therefore, from \eqref{eq-Omega-N}, \eqref{eq-N-G} and \eqref{eq-Omega-G}, we obtain the main theorem. Lemma \ref{lm-empty-int} shows that there exist infinitely many intervals $I$, for which the error of the asymptotic formula \eqref{eq-main} is of the order $O(Q^{n-1})$.

\section{About Theorem \ref{thm-limit}}\label{sec-limit}

\subsection{Quadratic algebraic integers}

In this subsection, we consider the desity function $\omega_2(\xi,t)$ of quadratic algebraic integers individually using its expression in elementary functions.
\begin{theorem}[\cite{Koleda2016}]\label{thm-omega2}
For $\xi\le 1/4$,
\begin{equation*}
\omega_2(\xi, t) = \begin{cases}
1 + 4 \xi^2 t^2, & |t| \le t_1,\\
\frac{1}{2t^2} + \frac12 + \xi (1-2|t|) + \frac52 \xi^2 t^2, & t_1 < |t| \le t_2,\\
\frac{1}{t^2} + \xi^2 t^2, & t_2 < |t| \le t_3,\\
2\xi, & t_3 < |t| \le t_4,\\
\frac{1}{2t^2} - \frac12 + \xi (1+2|t|) - \frac32 \xi^2 t^2, &  t_4 < |t| \le t_5,\\
0, & |t| > t_5.
\end{cases}
\end{equation*}
Here
\begin{align*}
&t_1 = t_1(\xi) = \frac{-1+\sqrt{1+4\xi}}{2\xi}, \quad t_2 = t_2(\xi) = \frac{1-\sqrt{1-4\xi}}{2\xi}, \quad t_3 = t_3(\xi) = \frac{1}{\sqrt{\xi}}, \\
&t_4 = t_4(\xi) = \frac{1+\sqrt{1-4\xi}}{2\xi}, \quad t_5 = t_5(\xi) = \frac{1+\sqrt{1+4\xi}}{2\xi}.
\end{align*}
\end{theorem}

Note that in \cite{Koleda2016} the expression (12) for $\omega_2(\xi,t)$ contains a typo: in the first case $\xi$~should be squared. Here we use the correct version.

From the general formula \eqref{eq-phi-def} we have
$\phi_1(t)=\frac{1}{\max(1,t^2)}$.
From Theorem~\ref{thm-omega2} one can see, the difference $|\omega_2(\xi, t)-\phi_1(t)|$ has the order $O(\xi)$ in a neighborhood of $t=1$ as $\xi\to 0$.
Whereas when $n\ge 3$ the estimate \eqref{eq-omega-psi} gives the magnitude $O(\xi^2)$ for $|\omega_n(\xi, t)-\phi_{n-1}(t)|$ in the same neighborhood. For all other $t$, except the two intervals $|t|\in (t_1,t_2)$, the general bound \eqref{eq-omega-psi} can be applied to the quadratic case too.

However, despite the uniform convergence of $\omega_2(\xi, t)$ to $\phi_1(t)$ as $\xi\to 0$,
the following theorem can be proved.
\begin{theorem}[\cite{Koleda2016}]
For $\xi\le 1/4$
\begin{equation*}
\int\limits_{-\infty}^{+\infty} \left(\omega_2(\xi, t) - \phi_1(t)\right) dt = 4 - \frac{16}{3} \sqrt{\xi} + O(\xi),
\end{equation*}
where the implicit constant in the big-O-notation is absolute.
\end{theorem}

\subsection{The proof of the limit equation for higher degrees}

Let $n\ge 3$. In this subsection, we assume $t\ge 0$ without loss of generality because the function $\omega_n(\xi,t)$ is even.

For the sake of simplification, introduce the following notation
\begin{align*}
\mathbf{p} &:= (p_1,\dots,p_{n-1}), \quad d\mathbf{p} := dp_1\,dp_2\dots dp_{n-1},\\
\mathbf{w}(t) &:= \left(t, t^2, \dots, t^{n-1}\right), \\
\mathbf{v}(t) &:= \left(1, 2t, \dots, (n-1)t^{n-2}\right) = \frac{d}{dt} \mathbf{w}(t).
\end{align*}

In this notation, the function $\phi_{n-1}(t)$ takes the form:
\[
\phi_{n-1}(t) = \int\limits_{G_{n-1}(t)} \left|\mathbf{v}(t)\cdot \mathbf{p}\right|\,d\mathbf{p}, \qquad t \in \mathbb{R},
\]
where
\begin{equation}\label{eq-G-def'}
G_{n-1}(t) = \left\{\mathbf{p} \in \mathbb{R}^{n-1} : \|\mathbf{p}\|_\infty \le 1, \ |\mathbf{w}(t)\cdot \mathbf{p}| \le 1 \right\}.
\end{equation}
Note that $G_{n-1}(t)$ differs from $D_n(\xi,t)$ in the absence of term $\xi t^n$ in the modulus brackets.

Change the variables in the integral for $\omega_n(\xi,t)$:
\[
\begin{cases}
p_i = q_i, & i=1,\dots,n-2,\\
p_{n-1} = q_{n-1} - \xi t.
\end{cases}
\]
Then the integral takes the form:
\[
\omega_n(\xi,t) = \int\limits_{S_n(\xi,t)} \left|\xi t^{n-1} + \sum_{k=1}^{n-1} k q_k t^{k-1}\right|\,dq_1\ldots\,dq_{n-1}, \qquad t \in \mathbb{R},
\]
where
\[
S_n(\xi,t) = \left\{(q_1,\ldots,q_{n-1}) \in \mathbb{R}^{n-1} :
\max\limits_{1\le i \le n-2}|q_i| \le 1, \
\begin{array}{l}
|q_{n-1}-\xi t| \le 1, \\
|q_{n-1} t^{n-1} + \ldots + q_1 t| \le 1
\end{array}
\right\}.
\]
Now, we partition the integral into the three summands
\begin{multline*}
\omega_n(\xi,t) = \int\limits_{G_{n-1}(t)} \left|\xi t^{n-1} + \mathbf{v}(t) \mathbf{q}\right|\,d\mathbf{q} \ +\\
+
\int\limits_{S_n^+(\xi,t)} \left|\xi t^{n-1} + \mathbf{v}(t) \mathbf{q}\right|\,d\mathbf{q} \  -
\int\limits_{S_n^-(\xi,t)} \left|\xi t^{n-1} + \mathbf{v}(t) \mathbf{q}\right|\,d\mathbf{q},
\end{multline*}
where $G_n(t)$ is defined in \eqref{eq-G-def}, and
\begin{align*}
S_n^+(\xi,t) &= \left\{ \mathbf{q} \in \mathbb{R}^{n-1} :
|\mathbf{w}(t) \mathbf{q}| \le 1, \ 
\max\limits_{1\le i \le n-2}|q_i| \le 1, \
\phantom{-}1 < q_{n-1} \le 1+\xi t
\right\},\\
S_n^-(\xi,t) &= \left\{ \mathbf{q} \in \mathbb{R}^{n-1} :
|\mathbf{w}(t) \mathbf{q}| \le 1, \
\max\limits_{1\le i \le n-2}|q_i| \le 1, \
-1 \le q_{n-1} < - 1+\xi t
\right\}.
\end{align*}
For convenience, denote the integral over $G_{n-1}(t)$ by $J_1$, that one over $S_n^+(\xi,t)$ by $J_2$, and that one over $S_n^-(\xi,t)$ by $J_3$. So
\[
\omega_n(\xi,t) = J_1 + J_2 - J_3.
\]

1) Estimation of the difference $J_2 - J_3$.

For $S_n(\xi,t)$ we have
\begin{equation}\label{eq-S-ineq}
\begin{cases}
-1+\xi t \le q_{n-1} \le 1+\xi t,\\
\frac{-1-q_1 t - \dots - q_{n-2} t^{n-2}}{t^{n-1}} \le q_{n-1} \le \frac{1-q_1 t - \dots - q_{n-2} t^{n-2}}{t^{n-1}}.
\end{cases}
\end{equation}
Obviously
\[
\min\limits_{\substack{|q_i|\le 1\\1\le i\le n-2}} \frac{-1-q_1 t - \dots - q_{n-2} t^{n-2}}{t^{n-1}} = -\sum_{k=1}^{n-1} t^{-k}.
\]
Therefore, the restriction $|q_{n-1} - \xi t|\le 1$ has no effect if $t$ satisfies the inequality
\begin{equation}\label{eq-t-cond}
-\sum_{k=1}^{n-1} t^{-k} \ge -1+\xi t.
\end{equation}

It is easily seen that for $0< \xi <1/8$, the inequality \eqref{eq-t-cond} defines an interval $[t_1, t_2]$ in the positive semiaxis, where $t_1$ and $t_2$ are the two positive roots of the equation
\begin{equation}\label{eq-t-eq}
1-\sum_{k=1}^{n-1} t^{-k} = \xi t.
\end{equation}
As $\xi$ tends to zero, these roots satisfy the asymptotics
\[
t_1 = t_1^{(0)} + O(\xi), \qquad t_2 = \xi^{-1} + O(1),
\]
where $t_1^{(0)}$ is the only positive root of \eqref{eq-t-eq} when $\xi=0$; the implicit big-O-constants depend on $n$ only.
Note that $t_2<\xi^{-1}$.

Consider separately the three intervals: $[0,t_1)$, $[t_1,t_2)$, $[t_2, +\infty)$.

a) Let $0\le t < t_1$.

In the integral $J_2$, make the change $q_{n-1}=1+\theta$. In the integral $J_3$, we make the change $\mathbf{q} \to -\mathbf{q}$ and next the change $q_{n-1} = 1-\theta$. Here $0<\theta<\xi t$.
Thus, we obtain
\begin{align*}
J_2 &= \int\limits_{\widetilde{S}_n^+(\xi,t)} \left|\xi t^{n-1} + (n-1)(1+\theta)t^{n-2} + \sum_{k=1}^{n-2} k q_k t^{k-1} \right|dq_1\,dq_2 \dots dq_{n-2}\,d\theta,\\
J_3 &= \int\limits_{\widetilde{S}_n^-(\xi,t)} \left|-\xi t^{n-1} + (n-1)(1-\theta)t^{n-2} + \sum_{k=1}^{n-2} k q_k t^{k-1} \right|dq_1\,dq_2 \dots dq_{n-2}\,d\theta,
\end{align*}
where
\begin{align*}
\widetilde{S}_n^+(\xi,t) &= \left\{ (q_1,\dots,q_{n-2},\theta) \in [-1,1]^{n-2}\times[0,\xi t] \ : \
\left|\sum_{k=1}^{n-2} q_k t^k + (1+\theta)t^{n-1}\right| \le 1
\right\},\\
\widetilde{S}_n^-(\xi,t) &= \left\{ (q_1,\dots,q_{n-2},\theta) \in [-1,1]^{n-2}\times[0,\xi t] \ : \
\left|\sum_{k=1}^{n-2} q_k t^k + (1-\theta)t^{n-1}\right| \le 1
\right\}.
\end{align*}

The measures of the symmetric difference and the intersection of $\widetilde{S}_n^-(\xi,t)$ and $\widetilde{S}_n^+(\xi,t)$ satisfy the inequalities
\begin{align*}
\mes_{n-1} \left(\widetilde{S}_n^-(\xi,t) \operatorname{\Delta} \widetilde{S}_n^+(\xi,t)\right) & \ll_n \xi^2 t^n,\\
\mes_{n-1} \left(\widetilde{S}_n^-(\xi,t) \cap \widetilde{S}_n^+(\xi,t)\right) & \ll_n \xi t.
\end{align*}
The integrands are bounded in the domains of integration.
The difference of the integrands in $\widetilde{S}_n^-(\xi,t) \cap \widetilde{S}_n^+(\xi,t)$ is at most $2n\xi t^{n-1}$. Therefore, for $t\in [0, t_1)$ we have
\[
|J_2 - J_3| \ll_n \xi^2 t^n.
\]

b) For $t\in [t_1,t_2)$ we have $S_n(\xi,t) = G_{n-1}(t)$, and so
\[
J_2 = J_3 = 0.
\]

c) For $t\ge t_2$, we have
\[
J_2 = 0, \quad 0\le J_3 \le J_1. 
\]
In this interval, the integral $J_3$ can be estimated as follows
\[
J_3 \ \le \ \left(\xi t^{n-1} + \sum_{k=1}^{n-1} k t^{k-1}\right) \mes_{n-1} S_n^-(\xi,t) \ \ll_n \ \xi + \frac{1}{t} \ \ll_n \ \xi.
\]
Looking at the system \eqref{eq-S-ineq} one can observe that $J_1 = J_3$ and so $\omega_n(\xi,t) = 0$, if
\[
\sum_{k=1}^{n-1} t^{-k} \le -1 + \xi t.
\]
This happens as $t \ge t_3$, where $t_3$ is the positive root of the equation
\[
1+\sum_{k=1}^{n-1} t^{-k} = \xi t,
\]
and can be estimated as
\[
\frac{1}{\xi} < t_3 < \frac{1}{\xi} + 1.
\]

2) Estimation of the difference $J_1 - \phi_{n-1}(t)$.

From Lemma \ref{lm-diff}, we have
\[
0 < J_1 - \phi_{n-1}(t) \le \xi t^{n-1} \mes_{n-1} U_n(\xi,t),
\]
where
\begin{equation}\label{eq-U-def}
U_n(\xi,t) = \left\{\mathbf{q}\in\mathbb{R}^{n-1} : \|\mathbf{q}\|_\infty \le 1, \ |\mathbf{w}(t) \mathbf{q}|\le 1, \ |\mathbf{v}(t) \mathbf{q}| \le \xi t^{n-1}\right\}.
\end{equation}
Denote by $\mathcal{S}$ the section of the <<stick>> $\left\{\mathbf{q}\in\mathbb{R}^{n-1}: |\mathbf{w}(t) \mathbf{q}|\le 1, \ |\mathbf{v}(t) \mathbf{q}| \le \xi t^{n-1} \right\}$ by the linear span of the vectors $\mathbf{w}(t)$ and $\mathbf{v}(t)$.

Using the identity
\begin{equation*}
\left(\sum_{i=1}^k a_i^2 \right) \left(\sum_{i=1}^k b_i^2 \right) - \left(\sum_{i=1}^k a_i b_i \right)^2 = \sum_{1\le i < j \le k} (a_i b_j - a_j b_i)^2,
\end{equation*}
we have
\[
\mathbf{w}(t)^2 \mathbf{v}(t)^2 - (\mathbf{w}(t) \mathbf{v}(t))^2 \ge t^{4(n-2)} + t^4 \ge
\left(\frac{t^{2(n-2)} + t^2}{2}\right)^2.
\]
From Lemma \ref{lm-mes-diam}, the diameter of the section can be estimated as
\[
\operatorname{diam} \mathcal{S} \asymp_n \frac{(t^{n-2}+1) + \xi t^{n-1} (t^{n-1}+t)}{t^{2(n-2)} + t^2}.
\]

a) For $0\le t\le 1/2$ the condition $|\mathbf{w}(t)\mathbf{q}|\le 1$ holds automatically and so is redundant. Thus, we have the estimate
\[
\mes_{n-1} U_n(\xi, t) \ll_n \xi t^{n-1}.
\]

b) Let $1/2 < t\le \kappa_1(n)/\sqrt{\xi}$. Here the upper bound for $t$ is determined by the condition that the diameter of the section does not exceed the diameter of the $n$-dimensional cube $[-1,1]^n$:
\[
\operatorname{diam} \mathcal{S} \ll_n 1.
\]
From Lemma \ref{lm-mes-diam}, we obtain
\[
\mes_{n-1} U_n(\xi,t) \ll_n \frac{\xi t^{n-3}}{t^{2(n-3)} + 1}.
\]

c) For $t > \kappa_1(n)/\sqrt{\xi}$, the estimate by Lemma \ref{lm-mes-diam} is not effective. So we esimate the measure in another way
\[
\mes_{n-1} U_n(\xi,t) \le \mes_{n-1}\left\{\mathbf{q}\in\mathbb{R}^{n-1} : \|\mathbf{q}\|_\infty \le 1, \ |\mathbf{w}(t) \mathbf{q}|\le 1 \right\} \ll_n t^{-n+1}.
\]

Gathering all the cases, we write
\[
|J_1 - \phi_{n-1}(t)| \ll_n \begin{cases}
\xi^2 t^{2(n-1)}, & 0\le t \le 1/2,\\
\xi^2 t^2, & 1/2 < t \le \frac{\kappa_1(n)}{\sqrt{\xi}},\\
\xi, & \frac{\kappa_1(n)}{\sqrt{\xi}} < t.
\end{cases}
\]
Since $\omega_n(\xi,t)=0$ for $t\ge \xi^{-1} + \xi$, we have
\[
|\omega_n(\xi, t) - \phi_{n-1}(t)| \ll_n \begin{cases}
\xi^2 t^2, & |t|\le \kappa_1(n) \xi^{-1/2},\\
\xi, & \kappa_1(n) \xi^{-1/2} < |t| \le \kappa_2(n)\xi^{-1},\\
t^{-2}, & \kappa_2(n)\xi^{-1} < |t|.
\end{cases}
\]
Theorem \ref{thm-limit} is proved.

\subsection{Proving Theorem \ref{thm-idiff}}

In this subsection we continue using the notation from the previous subsection.
Recall that
\[
J_1 = \int\limits_{G_{n-1}(t)} \left|\xi t^{n-1} + \mathbf{v}(t) \mathbf{q}\right|\,d\mathbf{q},
\]
the regions $G_{n-1}(t)$ and $U_n(\xi,t)$ are defined in \eqref{eq-G-def'} and \eqref{eq-U-def} respectively.

We start with the following fact.
\begin{lemma}\label{lm-I1}
Let $n\ge 3$ be a fixed integer, and $0<\xi<1$. Then for all
\begin{equation}\label{eq-t-restr}
|t|\ge \sqrt{\frac{5(n-1)}{\xi}}
\end{equation}
we have
\[
J_1 = 2^{n-1} \xi.
\]
\end{lemma}
\begin{proof}
In Lemma \ref{lm-diff}, the region $G_{n-1}(t)$ plays the role of~$V$, and $U_n(\xi,t)$ --- the role of~$V(\epsilon)$.
Now our aim is to determine values of $t$, for which $U_n(\xi,t)$ coincides with $G_{n-1}(t)$. For this sake consider the inequalities defining $G_{n-1}(t)$ and $U_n(\xi,t)$.

The inequality $|\mathbf{w}(t) \mathbf{q}|\le 1$ is equivalent to
\begin{equation}\label{eq-1ineq}
\left|q_{n-1} + \sum_{i=2}^{n-1} \frac{q_{n-i}}{t^{i-1}} \right|\le \frac{1}{t^{n-1}}.
\end{equation}

And the inequality $|\mathbf{v}(t) \mathbf{q}|\le \xi t^{n-1}$ can be rewritten as follows
\begin{equation}\label{eq-2ineq}
\left|q_{n-1} + \sum_{i=2}^{n-1} \frac{(n-i)q_{n-i}}{(n-1)t^{i-1}} \right|\le \frac{\xi t}{n-1}.
\end{equation}

Since $\max_{1\le i\le n-1}|q_i|\le 1$, both the sums in \eqref{eq-1ineq} and \eqref{eq-2ineq} containing $t$ don't exceed $2 t^{-1}$ by absolute value when $|t|\ge 2$.

Roughly estimating $|q_{n-1}|$ from \eqref{eq-1ineq}, we obtain
\begin{equation}\label{eq-q-bound}
|q_{n-1}|\le 3 t^{-1}.
\end{equation}
Putting this in \eqref{eq-2ineq} shows that the left hand side of \eqref{eq-2ineq} doesn't exceed $5 t^{-1}$.
Thus, for $|t|\ge c \xi^{-1/2}$, where $c = \sqrt{5(n-1)}$, the inequality \eqref{eq-2ineq} follows from \eqref{eq-1ineq},
and therefore,
\[
U_n(\xi,t) = G_{n-1}(t).
\]

By Lemma \ref{lm-diff} we have
\[
J_1 = \xi t^{n-1} \mes_{n-1} G_{n-1}(t).
\]

The bound \eqref{eq-q-bound} shows that if $|t|\ge 3$ then $G_{n-1}(t)$ can be written as
\[
G_{n-1}(t) = \left\{\mathbf{q} \in \mathbb{R}^{n-1} : \max_{1\le i\le n-2} |q_i| \le 1, \ 
\text{and \eqref{eq-1ineq} holds}
\right\}.
\]
Obviously,
\[
\mes_{n-1} G_{n-1}(t) = \frac{2^{n-1}}{t^{n-1}}.
\]
Note that if $\xi<1$, the inequality \eqref{eq-t-restr} is the most restrictive of the ones appearing in the proof.
The lemma is proved.
\end{proof}

Now we prove Theorem \ref{thm-idiff}. Consider integration over the three intervals: $[0,c\xi^{-1/2})$, $[c\xi^{-1/2}, t_2)$, $[t_2,+\infty)$,
where $t_2$ is the bigger positive root of \eqref{eq-t-eq}, which has the asymptotics $t_2 = \xi^{-1} + O(1)$ as $\xi\to 0$.

Using Theorem \ref{thm-limit} we find
\[
\int_0^{c \xi^{-1/2}} (\omega_n(\xi,t)-\phi_{n-1}(t))\,dt = O(\xi^{1/2}),
\]
\[
\int_{t_2}^{+\infty} (\omega_n(\xi,t)-\phi_{n-1}(t))\,dt = O(\xi).
\]
About $\phi_n(t)$ it is known \cite{Koleda2014a} that  $\phi_n(t) = O(t^{-2})$ for $t\ge 1$.
Hence, from Lemma \ref{lm-I1} we obtain
\[
\int_{c \xi^{-1/2}}^{t_2} (\omega_n(\xi,t)-\phi_{n-1}(t))\,dt = 2^{n-1} + O(\xi^{1/2}).
\]
Now remembering that both $\omega_n(\xi,t)$ and $\phi_n(t)$ are even functions with respect to $t$, we get Theorem~\ref{thm-idiff}.

\cleardoublepage
\phantomsection
\addcontentsline{toc}{section}{References}

\bibliographystyle{abbrv}
\bibliography{bib4-eng}

\bigskip

{\small\noindent Denis Koleda (Dzianis Kaliada)}\\
{\footnotesize
{Institute of Mathematics, National Academy of Sciences of Belarus,\\
220072 Minsk, Belarus}\\
e-mail: koledad@rambler.ru
}

\end{document}